\documentclass[reqno]{amsart}
\usepackage{amscd}
\usepackage[dvips]{graphics}

\usepackage{graphicx,subfigure,float,caption2,color,epstopdf}

\usepackage{amssymb}
\usepackage{mathrsfs}
\usepackage{amsmath}
\usepackage{amsthm}
\usepackage{bbm}
\usepackage{cases}
\usepackage{amsfonts}
\usepackage{stmaryrd}
\usepackage{mathtools}
\newtheorem{thm}{Theorem}[section]
\newtheorem*{thma}{Theorem A}
\newtheorem*{thmb}{Theorem B}
\newtheorem*{thmc}{Theorem C}
\newtheorem*{thmd}{Theorem D}
\newtheorem{cor}[thm]{Corollary}
\newtheorem{lem}[thm]{Lemma}
\newtheorem{prop}[thm]{Proposition}
\theoremstyle{definition}
\newtheorem{defn}[thm]{Definition}
\theoremstyle{remark}
\newtheorem{rem}[thm]{Remark}
\numberwithin{equation}{section}

\begin{document}

\title[CR Yamabe equation on the Heisenberg group]{CR Yamabe Equation on the Heisenberg Group\\
via the method of moving spheres}
\author{Congwen Liu}
\address{School of Mathematical Sciences, University of Science and Technology of China, Hefei, Anhui 230026, China}
\email {cwliu@ustc.edu.cn}
\thanks{
This work was supported by the National Natural Science Foundation of China grant 12371084.}


\date{}

\begin{abstract}
In this paper, we classify positive solutions to the CR Yamabe equation on the Heisenberg group. We show that all such solutions are Jerison-Lee bubbles, without imposing any finite-energy or a priori symmetry assumptions.
This result can be regarded as an analogue for $\mathbb{H}^n$ of the celebrated Caffarelli-Gidas-Spruck classification theorem in $\mathbb{R}^n$. To establish this, we develop a systematic approach to implement the method of moving spheres in the setting of the Heisenberg group.
\end{abstract}

\maketitle

\section{Introduction}
\label{ch:intro}

In this paper,  we consider positive solutions to the CR Yamabe equation
\begin{equation} \label{eqn:cr_yamabe}
-\Delta_{\! H} u = u^{\frac{Q+2}{Q-2}} \qquad \text{in } \mathbb{H}^n,
\end{equation}
where $\mathbb{H}^n$ is the Heisenberg group,  $u$ is a smooth, real and positive function defined in $\mathbb{H}^n$,  $\Delta_{\! H} u$ is the sub-Lapacian of $u$ (see the definition in Section 2) and
$Q:=2n+2$ is the homogeneous dimension of $\mathbb{H}^n$.

The equation \eqref{eqn:cr_yamabe} is closely related to the CR Yamabe problem in $\mathbb{H}^n$,
which states that: find a choice of contact form on $\mathbb{H}^n$ for which the pseudo-Hermitian 
scalar curvature is constant. Indeed, given the standard contact form $\Theta$ on $\mathbb{H}^n$, 
consider the conformal contact form $\theta=u^{\frac{2}{n}}\Theta$, then the pseudo-Hermitian 
scalar curvature associated to $\theta$ is the positive constant $R\equiv 4n(n+1)$ if and only if 
$u$ solves the equation
\begin{equation}\label{eqn:cr_yamabe2}
-\Delta_{\! H} u = 2n^2u^{\frac{Q+2}{Q-2}}. 
\end{equation}

In their celebrated paper \cite{JL88}, Jerison and Lee established the following classification of positive finite-energy solutions of \eqref{eqn:cr_yamabe}.

\begin{thma}[Jerison, Lee] \label{thm:JL}
Let $u \in C^2 \cap L^{\frac{2Q}{Q-2}}(\mathbb{H}^n)$ be a positive solution of \eqref{eqn:cr_yamabe}. Then
$u$ is of the form
\begin{equation}\label{eqn:JLbubble}
u(z,t) = K\bigl| t + i|z|^2 + \mu \cdot z + \kappa \bigr|^{-\frac{Q-2}{2}},
\end{equation}
where $K>0$, $\mu\in\mathbb{C}^n$, $\kappa\in\mathbb{C}$ with $\operatorname{Im}\kappa > \frac {|\mu|^2}{4}$.
\end{thma}

The functions in \eqref{eqn:JLbubble} are usually called \emph{Jerison–Lee bubbles}; they are the only extremals of 
the sharp Sobolev inequality on $\mathbb{H}^n$ and play the same role as the Aubin–Talenti bubbles in the original 
Yamabe problem. 

It has been a long‑standing open question whether the finite-energy assumption $u \in L^{\frac{2Q}{Q-2}}(\mathbb{H}^n)$
can be removed. Little progress has been made over decades, we only refer to \cite{GV01} where the authors 
obtained a Liouville result under the assumption of cylindrical symmetry on groups of Heisenberg type.

A major breakthrough was recently achieved by Catino, Li, Monticelli and Roncoroni \cite{CLMR23}, who provided a 
complete answer in the one-dimensional case and made significant progress in higher dimensions.

\begin{thmb}[Catino, Li, Monticelli, Roncoroni] \label{thm:CLMR}
\begin{enumerate}
\item[(i)] Let $u$ be a positive solution to \eqref{eqn:cr_yamabe} in $\mathbb{H}^1$. Then $u$ is a Jerison–Lee bubble.
\item[(ii)] Let $u$ be a positive solution to \eqref{eqn:cr_yamabe} in $\mathbb{H}^n$, $n\geq 2$ such that
\[
u(\xi)\leq\frac{C}{1+|\xi|^\frac{Q-2}{2}}\quad\forall\xi\in \mathbb{H}^n,
\]
for some $C>0$. Then $u$ is a Jerison–Lee bubble.
\end{enumerate}
\end{thmb}

Flynn and V\'etois \cite{FV23} made a further development for higher dimensions:

\begin{thmc}[Flynn, V\'etois] \label{thm:FV}
Let $n\ge2$ and $u$ be a positive solution to \eqref{eqn:cr_yamabe} such that
\begin{equation}\label{Th1Eq3}
u(z,t)\leq C \left(\left|z\right|^2+\left|t\right|\right)^{-\frac{n-2}{2}} 
\qquad \forall (z,t)\in\mathbb{H}^n\backslash\left\{(0,0)\right\},
\end{equation}
for some constant $C>0$. Then \(u\) is a Jerison–Lee bubble.
\end{thmc}

The main goal of the present paper is to establish the complete classification for all dimensions $n\geq 1$, 
without any decay or integrability assumptions. More precisely, our main result is the following theorem.

\begin{thm}[Main Theorem] \label{thm:main}
For $n \geq 1$, every positive solution $u$ to \eqref{eqn:cr_yamabe} is a Jerison–Lee bubble, that is,
\begin{equation*}
u(z,t) = K\bigl| t + i|z|^2 + \mu \cdot z + \kappa \bigr|^{-\frac{Q-2}{2}},
\end{equation*}
for some $K>0$, $\mu\in\mathbb{C}^n$, $\kappa\in\mathbb{C}$ with $\operatorname{Im}\kappa > \frac {|\mu|^2}{4}$.
\end{thm}

Proofs in both \cite{CLMR23} and \cite{FV23} rely on a generalization of the Jerison-Lee's differiential identity 
combined with integral estimates. We take a completely different approach.

Before presenting our approach, it is useful to recall the Euclidean analogue of equation \eqref{eqn:cr_yamabe2}:
\begin{equation}\label{eq:euclidean_yamabe}
-\Delta u = n(n-2)u^{\frac{n+2}{n-2}}, \quad \text{in } \mathbb{R}^n \; (n \ge 3).
\end{equation}
This equation is closely related to the Yamabe problem in Riemannian geometry and to the extremals of the sharp Sobolev inequality.

Using the method of moving planes, Gidas, Ni and Nirenberg \cite{GNN81} proved that any positive \(C^2\) solution of \eqref{eq:euclidean_yamabe} satisfying
\begin{equation}\label{eq:finite_condition}
\liminf_{|x| \to \infty} \left( |x|^{n-2} u(x) \right) < \infty,
\end{equation}
must be of the form
\[
u(x) = \left( \frac{a}{1 + a^2 |x - \bar{x}|^2} \right)^{\frac{n-2}{2}},
\]
where \(a > 0\) is a constant and \(\bar{x} \in \mathbb{R}^n\).

The hypothesis \eqref{eq:finite_condition} was later removed by Caffarelli, Gidas and Spruck \cite{CGS89}, an important advance for applications. Their proof again relied on the method of moving planes. Since then, the method has become a powerful tool in the study of nonlinear elliptic equations.

Li and Zhu \cite{LZ95} provided an alternative proof of the theorem of Caffarelli, Gidas and Spruck using the \textit{method of moving spheres}---a variant of the moving planes method that incorporates conformal invariance. This approach fully exploits the conformal invariance of the problem and yields the solutions directly, without first establishing radial symmetry and then classifying radial solutions. Significant simplifications to this proof were later given by Li and Zhang \cite{LZ03}.

A cornerstone of the method of moving spheres is a pair of calculus lemmas that translate geometric comparisons into an explicit formula for the solution. In the Euclidean case, such lemmas were first formulated by Li and Zhu (Lemmas 2.2 and 2.5 in \cite{LZ95}) and later refined by Li and Nirenberg (Lemmas 5.7 and 5.8 in \cite{Li04}). In this paper, we establish their counterparts in the Heisenberg group.

The starting point is the following definition.

\begin{defn}\label{def:gndcrinv}
For $\xi=(z',t')\in \mathbb{H}^n$, $\lambda>0$ and $\beta \in \mathbb{R}$, we define
the \emph{generalized CR inversion} $\Phi_{\xi, \lambda}^{\beta}: \mathbb{H}^n\setminus \{\xi\}\to \mathbb{H}^n\setminus \{\xi\}$ by
\begin{equation}\label{eqn:gndcrinv}
\Phi_{\xi, \lambda}^{\beta} := \tau_{\xi} \circ \rho_{M_{\xi,\beta}} \circ \delta_{\lambda^2} \circ
\mathscr{J} \circ \iota \circ \tau_{\xi}^{-1}
\end{equation}
where $M_{\xi,\beta} \in \mathcal{U}(n)$ is the diagonal matrix
\begin{equation}\label{eqn:Mxibeta}
M_{\xi,\beta} = \operatorname{diag} \left( e^{i\theta_1}, \cdots, e^{i\theta_n} \right)
\end{equation}
with $\theta_k$ for $k = 1, \ldots, n$ given by
\[
\theta_k := \begin{cases}
2 \operatorname{arg} z_k' + \operatorname{arg} \left(t'+i|z'|^2
+i\beta\right),& \text{if } z_k'\neq 0\\
0, & \text{if } z_k' = 0.
\end{cases}
\]
(See Section \ref{subsec:CRmaps} for the definitions of the CR maps $\tau_{\xi}$, $\rho_{M}$, $\delta_{\lambda^2}$, $\mathscr{J}$ and $\iota$.)
\end{defn}
The generalized CR inversion $\Phi_{\xi, \lambda}^{\beta}$ can be viewed as the CR inversion in the Kor\'anyi metric 
sphere $\partial B_{\lambda}(\xi)$, playing a role analogous to reflection across the Euclidean sphere $S(a, r)$ in $\mathbb{R}^n$.

This definition draws inspiration from the pioneering attempt of Han, Wang, and Zhu \cite{HWZ17} and is justified by the proposition below.

\begin{prop}\label{prop:justifydefn}
Let $\nu, \beta \in \mathbb{R}$ and define
\[
f_{\beta}(z,t) := \bigl| t + i|z|^{2} + i\beta \bigr|^{-\nu/2}.
\]
Then for every $\xi \in \mathbb{H}^{n}$, there exists $\lambda(\xi) > 0$ such that
\begin{equation}\label{eqn:funceqn}
\left(\frac{\lambda(\xi)}{d_{H}\!\left(\xi,\zeta\right)}\right)^{\!\nu} 
\, f_{\beta}\Bigl(\Phi_{\xi,\lambda(\xi)}^{\beta}(\zeta)\Bigr) 
\;=\; f_{\beta}(\zeta), \qquad \forall\, \zeta \in \mathbb{H}^{n}\setminus\{\xi\}.
\end{equation}
\end{prop}

See Section \ref{subsec:Pf_justified} for the proof of the proposition.

We are now ready to state our calculus lemmas of Li–Nirenberg–Zhu type in the Heisenberg setting.

\begin{thm}[Calculus Lemma I of Li-Nirenberg-Zhu type] \label{thm:lnzlem1}
Let \(n\ge 1\) and \(\nu,\beta\in\mathbb{R}\). Assume \(f:\mathbb{H}^n\to\mathbb{R}\) satisfies
\begin{equation}\label{eqn:lnzinq}
\Bigl(\frac{\lambda}{d_H(\xi,\zeta)}\Bigr)^{\nu}
\,f \bigl(\Phi_{\xi,\lambda}^{\beta}(\zeta)\bigr) \leq f(\zeta)
\qquad \forall \xi\in\mathbb{H}^n,\; \forall \zeta\in\Sigma_{\lambda}(\xi),\; \forall \lambda>0,
\end{equation}
where \(\Sigma_{\lambda}(\xi):=\mathbb{H}^n\setminus \overline{B_\lambda(\xi)}\). Then \(f\) is constant.
\end{thm}

\begin{thm}[Calculus Lemma II of Li-Nirenberg-Zhu type] \label{thm:lnzlem2}
Let \(n\ge 1\) and \(\nu>0\). Suppose a nonnegative function \(f\in C^0(\mathbb{H}^n)\) attains its maximum at the origin, and for every \(\xi\in\mathbb{H}^n\) there exists \(\lambda(\xi)>0\) such that
\[
\Bigl(\frac{\lambda(\xi)}{d_H(\xi,\zeta)}\Bigr)^{\nu}
f\!\bigl(\Phi_{\xi,\lambda(\xi)}^{\beta_f}(\zeta)\bigr)=f(\zeta)
\qquad \forall \zeta\in\mathbb{H}^n\setminus\{\xi\},
\]
where 
\begin{equation}\label{eqn:alphabeta}
\beta_f:=(\alpha_f)^{2/\nu}f(0)^{-2/\nu} \quad \text{and} \quad \alpha_f:=\lim_{|\zeta|_H\to\infty}|\zeta|_H^{\nu}f(\zeta).
\end{equation}
Then
\[
f(z,t) = \alpha_f\,\bigl| t + i|z|^2 + i\beta_f \bigr|^{-\frac{\nu}{2}}.
\]
\end{thm}

With these two calculus Lemmas of Li-Nirenberg-Zhu type in hand, combined with a Terracini-type integral 
inequality (see Section \ref{sec:Terracini}), the proof of Theorem \ref{thm:main} follows the same scheme as 
was used by Li and Zhang in Section 2 of \cite{LZ03} (also in \cite{LZ95}). There are, broadly speaking, 
three main steps in this proof: first, to initiate the moving spheres procedure starting from small radii; 
second, to show that if the procedure stops, the solution must coincide with its generalized Kelvin 
transform; and third, to apply our calculus lemmas to demonstrate that the solution necessarily takes 
the Jerison–Lee profile. The complete argument is presented in Section \ref{sec:PfofMainThm}.

Our method is equally applicable to the subcritical case.
Consider the following subcritical equation
\begin{equation}\label{eqn:subcritical}
-\Delta_{\! H} u = u^{p} \qquad \text{in } \mathbb{H}^n
\end{equation}
where $1<p < \frac {Q+2}{Q-2}$.
  
Birindelli et al. \cite{BCC97} investigated the case $1<p<\frac {Q}{Q-2}$ and proved that the unique 
nonnegative solution of \eqref{eqn:subcritical} is $u=0$. Later, Xu \cite{Xu09} improved this result 
to the case $1<p < \frac {Q(Q+2)}{(Q-1)^2}$ ($n>1$). Recently, Ma and Ou \cite{MO23} extended the 
Liouville result to the whole interval of subcritical values of $p$.

\begin{thmd}[Ma, Ou] \label{thm:subcritical}
Let \(1<p<\frac{Q+2}{Q-2}\). Then the equation \eqref{eqn:subcritical}
admits no positive solution, namely, any nonnegative solution of \eqref{eqn:subcritical} must be identically zero.
\end{thmd}

We will recover this result using the method of moving spheres in Section \ref{sec:PfofMaOu}.

\section{Preliminaries}
\label{sec:preli}

\subsection{The Heisenberg Group}
\label{subsec:Heisenberg}

The Heisenberg group $\mathbb{H}^{n}$ is the Lie group whose underlying manifold is
\[
\mathbb{C}^{n} \times \mathbb{R}=\left\{(z, t): z=\left(z_{1}, \cdots, z_{n}\right) \in \mathbb{C}^{n}, t \in \mathbb{R}\right\},
\]
endowed with the group law: given $\xi=(z,t)$ and $\zeta=\left(z', t'\right)$,
\begin{equation}\label{eqn:grouplaw}
(z, t) \cdot \left(z', t'\right):=\left(z+z', t+t'+2 \operatorname{Im} \langle z, z'\rangle \right),
\end{equation}
where $\langle z, z'\rangle :=\sum_{j=1}^{n} z_{j} \overline{z'_j}$ is the hermitian inner product on $\mathbb{C}^n$. Haar measure on $\mathbb{H}^n$ is the usual Lebesgue measure $d\xi=dz\,dt$. (To be more precise, $dz=dx\,dy$ if $z=x+iy$ with $x,y\in\mathbb{R}^n$.) 

We define, for $\xi=(z,t)\in\mathbb{H}^n$, the Kor\'anyi norm
\[
|\xi|_{\! H} :=\left( \vert z\vert^4+ t^2\right)^{\frac{1}{4}},
\]
with the associated distance function
$$
d_{\! H}(\xi,\zeta) := \left\vert \zeta^{-1}\cdot \xi \right\vert_{\! H} \quad \text{for } \xi, \zeta\in\mathbb{H}^n\, ,
$$
where $\zeta^{-1}$ denotes the inverse of $\zeta$ with respect to the group law \eqref{eqn:grouplaw}, 
i.e.  $\zeta^{-1}=-\zeta$.

We use the notation $B_{\lambda}(\xi)$ for the metric ball centred at $\xi\in\mathbb{H}^n$ with radius $\lambda>0$, i.e.
\[
B_{\lambda}(\xi) :=\lbrace \zeta\in\mathbb{H}^n \, : \, d_{\! H}(\xi, \zeta)< \lambda\rbrace\,.
\]
Also, we write $\Sigma_{\lambda}(\xi):=\mathbb{H}^n\setminus \overline{B_{\lambda}(\xi)}$.

If we set \(z_j = x_j + iy_j\), $j=1,\ldots,n$, then \((x_1, \cdots, x_n, y_1, \cdots, y_n, t)\) form a real coordinate system for \(\mathbb{H}_n\). In this coordinate system we define the following vector fields:
\[
X_j := \frac{\partial}{\partial x_j} + 2y_j \frac{\partial}{\partial t}, \quad 
Y_j := \frac{\partial}{\partial y_j} - 2x_j \frac{\partial}{\partial t}, \quad 
T := \frac{\partial}{\partial t}, \quad j=1,\ldots,n.
\]
They form a basis for the left-invariant vector fields on \(\mathbb{H}_n\). 
The sub‑Laplacian (or Heisenberg Laplacian) on $\mathbb{H}^n$ is then defined by
\[
\Delta_{\! H}:=\sum_{j=1}^{n}\big(X_j^2+Y_j^2\big),
\]
and the horizontal gradient of a regular function $u:\mathbb{H}^n\to \mathbb{R}$ is defined by
\[
\nabla_{\! H} {u}:=(X_1u,\dots,X_nu,Y_1u,\dots,Y_nu).
\]

\subsection{CR maps and the generalized Kelvin transform on $\mathbb{H}^n$}
\label{subsec:CRmaps}

For any fixed $\xi=(z',t')\in\mathbb{H}^n$ we will denote by
$\tau_{\xi}:\mathbb{H}^n\rightarrow\mathbb{H}^n$ the \emph{left translation}
on $\mathbb{H}^n$ by $\xi$, defined by
\begin{equation}\label{trasl}
\tau_{\xi}(\zeta)=\xi\cdot\zeta, \quad \zeta\in \mathbb{H}^n,
\end{equation}
where $\cdot$ denotes the group law defined in \eqref{eqn:grouplaw},
while for any $\lambda>0$ we will denote by
$\delta_\lambda:\mathbb{H}^n\rightarrow\mathbb{H}^n$ the \emph{dilation} defined
by
\begin{equation}\label{dil}
\delta_\lambda(z,t):=(\lambda z,\, \lambda^2 t), \qquad (z,t)\in\mathbb{H}^n,
\end{equation}
which satisfies
$$\delta_\lambda(\xi\cdot\zeta)=\delta_\lambda(\xi)\cdot\delta_\lambda(\zeta)$$
for every $\xi$, $\zeta\in\mathbb{H}^n$ and every $\lambda>0$.

Notice that the Kor\'anyi norm is
homogeneous of degree $1$ with respect to the dilations
$\delta_\lambda$, i.e.
\begin{equation}\label{eqn:dilationnorm}
|\delta_\lambda(\zeta)|_{\! H}=\lambda|\zeta|_{\! H} \qquad\forall\,\zeta\in\mathbb{H}^n,\; \forall\lambda>0.
\end{equation}
For any unitary matrix $M\in \mathcal{U}(n)$, we will denote by
$\rho_{\! M}:\mathbb{H}^n\rightarrow\mathbb{H}^n$ the \emph{rotation} defined
by
\begin{equation}\label{rot}
\rho_{\! M}(z,t):=(Mz,t), \qquad (z,t)\in\mathbb{H}^n.
\end{equation}

We finally introduce the \emph{inversion} map
$\iota:\mathbb{H}^n\rightarrow\mathbb{H}^n$ defined by
\begin{equation}\label{invers}
\iota(z,t):=(\overline{z},-t), \qquad (z,t)\in\mathbb{H}^n
\end{equation}
and the map
$\mathscr{J}:\mathbb{H}^n\rightarrow\mathbb{H}^n$ defined by Jerison and Lee in
\cite{JL88} which we shall refer to as the \emph{CR inversion}
and which is given by
\begin{equation}\label{CRinvers}
\mathscr{J}(z,t):=\left(\frac {z}{w},-\frac {t}{|w|^2}\right), \qquad (z,t)\in\mathbb{H}^n\setminus \{(0,0)\},
\end{equation}
where $w:=t+i|z|^2$. We remark that
\begin{equation}\label{eqn:inv_norm}
\left|\mathscr{J}(\xi)\right|_{\! H}=|\xi|_{\! H}^{-1}. 
\end{equation}
The CR inversion in $\mathbb{H}^n$ plays the role of the usual inversion with respect to the
unitary sphere in $\mathbb{R}^n$.

A \emph{CR maps} on $\mathbb{H}^n$ is a finite composition of  
the left translations \eqref{trasl}, the dilations \eqref{dil},
the rotations \eqref{rot}, the inversion map \eqref{invers}
and the CR inversion \eqref{CRinvers}.


Following \cite{LM12}, for $u\in C^2(\mathbb{H}^n)$, we define the transformed function $u_\psi$ of $u$ through 
the CR map $\psi:\mathbb{H}^n\to \mathbb{H}^n$ by 
\begin{equation}
u_\psi(\xi):=|J_\psi(\xi)|^\frac{Q-2}{2Q}u\big(\psi(\xi)\big),\qquad\xi\in\mathbb{H}^n,
\end{equation}
where $J_\psi(\xi)$ denotes the Jacobian matrix of $\psi$ evaluated at
$\xi$ and $|J_\psi(\xi)|$ denotes its determinant.
In particular, the transformed function of $u$ under 
the generalized CR inversions $\Phi_{\xi, \lambda}^{\, \beta}$
(as we introduced in Definition \ref{def:gndcrinv})
is called the \emph{generalized Kelvin Transform} of $u$.
More precisely, we introduce the following definition.

\begin{defn}[Generalized Kelvin Transform]
Let $\xi\in\mathbb{H}^n$, $\lambda>0$ and $\beta\in\mathbb{R}$. For a function $u:\mathbb{H}^n\to\mathbb{R}$ 
we define its generalized Kelvin transform $u_{\xi,\lambda}^{\beta}$ by
\[
u_{\xi,\lambda}^{\beta}(\eta):=\Bigl(\frac{\lambda}{d_H(\eta,\xi)}\Bigr)^{Q-2}
u \bigl(\Phi_{\xi,\lambda}^{\beta}(\eta)\bigr),\qquad \eta\in\mathbb{H}^n\setminus\{\xi\}.
\]
\end{defn}

The following observation by Li and Monticelli \cite[Remark 2.9]{LM12}, is fundamental:

\begin{lem}\label{lem:Li-Monticelli}
For every $u\in C^2(\mathbb{H}^n)$ and every CR map $\psi:\mathbb{H}^n\rightarrow\mathbb{H}^n$,
one has
\begin{equation}
u_\psi^{-\frac{Q+2}{Q-2}}\Delta_{\! H}u_\psi=\big(u^{-\frac{Q+2}{Q-2}}\Delta_{\! H}u\big)\circ\psi.
\end{equation}
\end{lem}

\begin{cor}\label{cor:conformal_invariance}
If $u\in C^2(\mathbb{H}^n)$ solves \eqref{eqn:cr_yamabe}, then for any $\xi,\lambda,\beta$ the 
generalized Kelvin transform $u_{\xi,\lambda}^{\beta}$ also satisfies \eqref{eqn:cr_yamabe} 
on $\mathbb{H}^n\setminus\{\xi\}$. Moreover, 
if $u\in C^2(\mathbb{H}^n)$ solves
the equation
\[
-\Delta_{\! H} u = u^p \qquad \text{in } \mathbb{H}^n,
\]
then $u_{\xi,\lambda}^{\beta}$ satisfies
\[
-\Delta_{\! H} u_{\xi,\lambda}^{\beta} (\zeta) = \left( \frac{\lambda}{ d_{\! H} (\xi, \zeta) } \right)^{(Q+2) - (Q-2)p} 
\bigl(u_{\xi,\lambda}^{\beta} (\zeta)\bigr)^p \qquad \text{in } \mathbb{H}^n\setminus \{\xi\}.
\]
\end{cor}

\subsection{Basic properties of the generalized CR inversion $\Phi_{\xi, \lambda}^{\, \beta}$}
\label{subsec:Basicproperties}

\begin{lem}\label{lem:gndCRinv1}
For any $\xi\in \mathbb{H}^n$ and any $\zeta \in \mathbb{H}^n\setminus \{\xi\}$,
\begin{equation}\label{eqn:reflectionidentity}
d_{\! H}\! \left(\Phi_{\xi, \lambda}^{\, \beta} (\zeta), \xi\right) \; d_{\! H}\! (\zeta, \xi)=\lambda^{2}.
\end{equation}
\end{lem}

\begin{proof}
In view of \eqref{eqn:dilationnorm} and \eqref{eqn:inv_norm}, we have
\begin{equation*}\label{eqn:reflectionform}
d_{\! H}\! \left(\Phi_{\xi, \lambda}^{\, \beta} (\zeta), \xi\right)
= \bigl|\delta_{\lambda^{2}} \circ \mathscr{J} \circ \tau_{\xi}^{-1}(\zeta)\bigr|_{\! H}
= \lambda^{2}\,\bigl|\mathscr{J} (\xi^{-1}\cdot\zeta)\bigr|_{\! H}
= \frac{\lambda^{2}}{\bigl|\xi^{-1}\cdot\zeta\bigr|_{\! H}},
\end{equation*}
which is precisely the identity \eqref{eqn:reflectionidentity}.
\end{proof}

\begin{cor}\label{cor:gndCRinv1}
For fixed $\lambda>0$ and $\beta\in \mathbb{R}$, we have $\Phi_{\xi,\lambda}^{\beta} (\zeta) \to \xi$ as $|\zeta|_{\! H}\to \infty$.
\end{cor}

\begin{cor}\label{cor:gndCRinv2}
$\Phi_{\xi, \lambda}^{\, \beta} \left(B_{\lambda}(\xi)\right)=\Sigma_{\lambda}(\xi)$ and $\Phi_{\xi, \lambda}^{\, \beta} \left(\Sigma_{\lambda}(\xi)\right)=B_{\lambda}(\xi)$.
\end{cor}

\begin{lem}\label{lem:gndCRinv2}
$\Phi_{\xi, \lambda}^{\, \beta}$ is an involution, i.e., 
\[
\bigl( \Phi_{\xi,\lambda}^{\beta} \circ \Phi_{\xi,\lambda}^{\beta} \bigr)(\zeta) ~=~ \zeta 
\qquad\text{for all } \zeta\in \mathbb{H}^n\setminus \{\xi\}.
\]
\end{lem}

\begin{proof}
Write $\Phi_{0,\lambda} := \delta_{\lambda^2} \circ \mathscr{J} \circ \iota$ and note that
\[
\Phi_{\xi,\lambda}^{\beta} = \tau_{\xi} \circ \rho_{M_{\xi,\beta}} \circ \Phi_{0,\lambda} \circ \tau_{\xi}^{-1}.
\]
Direct verifications show that
\[
\rho_{M_{\xi,\beta}} \circ \Phi_{0,\lambda} = \Phi_{0,\lambda} \circ \rho_{M_{\xi,\beta}}^{-1},
\]
and 
\[
\bigl( \Phi_{0,\lambda} \circ \Phi_{0,\lambda} \bigr)(\zeta) = \zeta \quad \forall \zeta \in \mathbb{H}^n \setminus \{0\}.
\]
It follows that
\begin{align*}
\bigl( \Phi_{\xi,\lambda}^{\beta} \circ \Phi_{\xi,\lambda}^{\beta} \bigr)(\zeta)
&= \Bigl( \tau_{\xi} \circ \rho_{M_{\xi,\beta}} \circ \Phi_{0,\lambda} \circ \rho_{M_{\xi,\beta}} \circ \Phi_{0,\lambda} \circ \tau_{\xi}^{-1} \Bigr)(\zeta) \\
&= \Bigl( \tau_{\xi} \circ \rho_{M_{\xi,\beta}} \circ \Phi_{0,\lambda} \circ \Phi_{0,\lambda} \circ \rho_{M_{\xi,\beta}}^{-1} \circ \tau_{\xi}^{-1} \Bigr)(\zeta) = \zeta
\end{align*}
for every $\zeta \in \mathbb{H}^n \setminus \{\xi\}$.
\end{proof}

\subsection{Proof of Proposition \ref{prop:justifydefn}}
\label{subsec:Pf_justified}

Fix $\xi=(z',t')\in \mathbb{H}^n$. A simple calculation show that
\begin{equation}\label{eqn:fbetacomptau}
\left( f_\beta \circ \tau_{(z',t')}\right)(z,t)
=\left|w +  2i\langle z, z' \rangle + w' + i\beta\right|^{-\nu/2}
\end{equation}
holds for every $(z,t)\in \mathbb{H}^n\setminus \{(0,0)\}$, where
\[
w:=t + i|z|^2  \quad \text{and} \quad w':= t' + i\left|z'\right|^2.
\]
It follows that
\begin{align*}
& \left(\frac{\lambda^{2}}{|w|}\right)^{\nu/2} 
\left(f_\beta \circ \tau_{(z',t')} \circ \rho_{M_{\xi,\beta}} \circ \delta_{\lambda^{2}}
\circ \mathscr{J} \circ \iota \right) (z,t)\\
& \quad = \left(\frac{\lambda^{2}}{|w|}\right)^{\nu/2} \left( f_\beta \circ \tau_{(z',t')}\right)
\left(\frac {\lambda^2 M_{\xi,\beta} \overline{z}} {-\overline{w}},
\frac {\lambda^4 t} {|w|^2} \right)\\
& \quad =  \left(\frac{\lambda^{2}}{|w|}\right)^{\nu/2} \left| \frac {\lambda^4 t} {|w|^2} +
i\left| \frac {\lambda^2 M_{\xi,\beta} \overline{z}} {-\overline{w}}\right|^{2}
+2i \left\langle \frac {\lambda^2 M_{\xi,\beta} \overline{z}} {-\overline{w}}, z'\right\rangle 
+ w' + i\beta \right|^{-\frac{\nu}{2}}\\
&\quad = \left(\frac{\lambda^{2}}{|w|}\right)^{\nu/2}
\left| \frac{\lambda^{4}}{\overline{w}} - \frac{2 \lambda^{2} i}{\overline{w}}
\left\langle \overline{z}, M_{\xi,\beta}^{-1} z' \right\rangle + w' 
+ i\beta\right|^{-\frac{\nu}{2}}\\
&\quad =\left|\lambda^{2}- 2 i\, \left\langle \overline{z}, M_{\xi,\beta}^{-1} z' \right\rangle + \frac{\overline{w}}{\lambda^{2}} \left(w' + i\beta\right)\right|^{-\frac{\nu}{2}} \\
&\quad =\left|\lambda^{2} + 2 i\, \left\langle z, M_{\xi,\beta} \overline{z'} \right\rangle + \frac{w}{\lambda^{2}} \left(\overline{w'} - i\beta\right)\right|^{-\frac{\nu}{2}} \\
&\quad =\left(\frac{\lambda^{2}}{\left|\overline{w'}-i\beta\right|}\right)^{\frac{\nu}{2}}
\left| w + 2i \,\left\langle z, \frac {\lambda^{2}}{w'+i\beta} M_{\xi,\beta} \overline{z'} \right\rangle + \frac{\lambda^4}{\overline{w'}-i\beta}\right|^{-\frac{\nu}{2}}
\end{align*}
Taking $\lambda=\lambda\left(\xi\right)=\left|\overline{w'}-i\beta\right|^{1/2}$, we obtain
\begin{align}\label{eqn:fbeta1}
& \left(\frac{\lambda^{2}}{|w|}\right)^{\nu/2} 
\left(f_\beta \circ \tau_{(z',t')} \circ \rho_{M_{\xi,\beta}} \circ \delta_{\lambda^{2}}
\circ \mathscr{J} \circ \iota \right) (z,t) \notag\\
& \qquad \quad ~=~\left| w + 2i \left\langle z, \frac{\left|w'+i\beta\right|}{w'+i\beta} 
M_{\xi,\beta} \overline{z'} \right\rangle w' + i\beta \right|^{-\frac{\nu}{2}}.
\end{align}
Recall that the unitary matrix $M_{\xi,\beta}$ is defined as in \eqref{eqn:Mxibeta}, so that
\[
\frac{\left|w'+i\beta\right|}{w'+i\beta} M_{\xi,\beta} \overline{z'} = z'.
\]
Plugging this into \eqref{eqn:fbeta1}, together with \eqref{eqn:fbetacomptau}, we obtain
\[
\left(\frac{\lambda^{2}}{|(z,t)|_{\! H}^2 }\right)^{\nu/2} 
\left(f_\beta \circ \tau_{(z',t')} \circ \rho_{M_{\xi,\beta}} \circ \delta_{\lambda^{2}}
\circ \mathscr{J} \circ \iota \right) (z,t) 
=\left( f_\beta \circ \tau_{(z',t')}\right)(z,t).
\]
Replacing $(z,t)$ with $\tau_{(z',t')}^{-1} (z,t)$ in the above yields
\[
\left(\frac{\lambda}{\big|\tau_{(z',t')}^{-1} (z,t) \big|_{\! H} }\right)^{\nu} 
f_\beta \left(\Phi_{(z',t'), \lambda(z',t')}^{\, \beta} (z,t)\right)
=f_\beta (z,t),
\]
which completes the proof.

\section{Proof of Theorem \ref{thm:lnzlem1}}
\label{sec:PfofLNZ1}

We first prove that for every $\zeta \neq \eta$ in $\mathbb{H}^{n}$ and $\lambda > 0$, there exists $\xi^{\ast} = \xi^{\ast}(\lambda) \in \mathbb{H}^n$ such that $\eta = \Phi_{\xi^{\ast},\lambda}^{\,\beta}(\zeta)$.

Consider the map
\[
T(\xi) := \eta \cdot \left( \Phi_{\xi,\lambda}^{\, \beta} (\zeta)\right)^{-1} \cdot \xi. 
\]
A direct estimate gives
\[
|T(\xi)|_{\! H} ~\leq~ |\eta|_{\! H} + \frac {\lambda^2}{d_{\! H}\! (\xi,\zeta)}.
\]
so $|T(\xi)|_{\! H} \leq R$ holds for all $\xi$ with $|\xi|_{\! H} = R$, provided $R$ is sufficiently large. Applying the Leray-Schauder fixed point theorem, we obtain a fixed point $\xi^{\ast} \in \overline{B_R(0)}$ for $T$, from which the desired identity $\eta = \Phi_{\xi^{\ast},\lambda}^{\,\beta}(\zeta)$ follows.

Then \eqref{eqn:lnzinq} becomes
\[
\bigg(\frac{\lambda}{d_{\! H}\big (\xi^{\ast},\zeta\big)}\bigg)^{\nu} f(\eta) ~\leq~ f(\zeta)
\qquad \forall \lambda>0.
\]
In view of \eqref{eqn:reflectionidentity}, we have
\[
\lim_{\lambda\to \infty} \frac{\lambda}{d_{\! H}\! \big(\xi^{\ast},\zeta\big)} ~=~
\lim_{\lambda\to \infty} \sqrt{\frac{d_{\! H}\! \big(\xi^{\ast},\eta\big)}
{d_{\! H}\! \big(\xi^{\ast},\zeta\big)}} ~=~ 1,
\]
and hence $f(\eta)\leq f(\zeta)$.
Theorem \ref{thm:lnzlem1} follows since $\eta \neq \zeta$ are arbitrary.

\section{Preparations for Proof of Theorem \ref{thm:lnzlem2}}
\label{sec:prePfofLNZ2}

\begin{prop}\label{pro:lnzlem2}
Let $n \geq 1$ and $\beta, \nu \in \mathbb{R}$. Suppose that $f \in C^{0}(\mathbb{H}^{n})$ satisfies: for every $\xi \in \mathbb{H}^{n}$, there exists $\lambda(\xi)>0$ such that
\begin{equation}\label{eqn:lnzfunceqn}
\left(\frac{\lambda(\xi)}{d_{\! H} \!\left(\xi,\zeta\right)}\right)^{\nu} f\left(\Phi_{\xi,\lambda(\xi)}^{\beta}(\zeta)\right) 
~=~ f(\zeta), \qquad \forall \zeta \in \mathbb{H}^{n}\setminus\{\xi\}.
\end{equation}
Then 
\begin{equation}\label{eqn:alpha0}
\alpha_f :=\lim _{|\zeta|_{\! H} \rightarrow \infty} |\zeta|_{\! H}^{\nu} f(\zeta) 
\end{equation}
exists and
\begin{equation}\label{eqn:alphaf1}
\lambda(\xi)^{\nu} f(\xi) = \alpha_f, \quad \forall \xi \in \mathbb{H}^{n} . 
\end{equation}
\end{prop}

\begin{rem}\label{rmk:lnzlem2}
Note that \eqref{eqn:alpha0} and \eqref{eqn:alphaf1} imply that the function $\xi \mapsto \lambda(\xi)$ is independent of $\beta$. Also, since $\alpha_f>0$, we find that the symmetry \eqref{eqn:lnzfunceqn} implies 
that $f(\zeta)$ decays as $O(|\zeta|_{\! H}^{-\nu})$ asymptotically. 

\end{rem}

\begin{proof}
Rewrite \eqref{eqn:lnzfunceqn} as
\begin{equation}\label{eqn:lnzfunceqn2}
|\zeta|_{\! H}^{\nu} f(\zeta) = \left(\frac{|\zeta|_{\! H}}{d_{\! H}\! \left(\xi,\zeta\right)}\right)^{\nu} 
\lambda(\xi)^{\nu} f\left(\Phi_{\xi,\lambda(\xi)}^{\beta}(\zeta)\right). 
\end{equation}
Note by Corollary \ref{cor:gndCRinv1} that $\Phi_{\xi,\lambda(\xi)}^{\beta} (\zeta) \to \xi$ as $|\zeta|_{\! H}\to \infty$. Thus, by the continuity of $f$,
\begin{equation}
\alpha_f:=\lim _{|\zeta|_{\! H} \rightarrow \infty} |\zeta|_{\! H}^{\nu} f(\zeta) =\lambda(\xi)^{\nu} f(\xi), \quad \forall \xi \in \mathbb{H}^{n}.
\end{equation}
\end{proof}

\begin{lem}\label{lem:fixedpoint1}
Let $f$ be the function from Theorem~\ref{thm:lnzlem2} and $\alpha_f$, $\beta_f$ be defined as in \eqref{eqn:alphabeta}.
For any $\epsilon>0$, there exists $C_{\epsilon}$ such that for any $|\zeta|_{\! H} \geq C_{\epsilon}$, there exists a point $\xi^{\ast}=\xi^{\ast}(\zeta) \in \mathbb{H}^n$ satisfying
\begin{equation}\label{fixedpoint1}
\Phi_{\xi^{\ast}, \lambda(\xi^{\ast})}^{\beta_f}(\zeta) = 0 
\quad \text { and } \quad \big|\xi^{\ast}\big|_{\! H} \leq \epsilon. 
\end{equation}
\end{lem}

\begin{proof}
We know from \eqref{eqn:alphaf1} that $\lambda(\xi)= \left(\alpha_f\right)^{\frac {1}{\nu}} f(\xi)^{-\frac {1}{\nu}}$ for all $\xi \in \mathbb{H}^{n}$. For any $\epsilon \in(0,1)$, pick $C_{\epsilon}>1$ so that
\[
\frac {\left(\alpha_f\right)^{\frac {2}{\nu}}} {C_{\epsilon}-\epsilon}\, \max_{|\xi|_{\! H} \leq \epsilon} f(\xi)^{-\frac {2}{\nu}} < \epsilon.
\]
For fixed $\zeta$ with $|\zeta|_{\! H} \geq C_{\epsilon}$, we consider the map $T:\overline{B_{\epsilon}(0)} \to 
\mathbb{H}^n$ given by
\[
T(\xi):= \left( \Phi_{\xi,\lambda(\xi)}^{\beta_f} (\zeta)\right)^{-1} \cdot \xi
\]
We have
\begin{align*}
\max_{|\xi|_{\! H} \leq \epsilon} \left|T(\xi)\right|_{\! H} 
~=~ \max_{|\xi|_{\! H} \leq \epsilon} \frac {\lambda(\xi)^2} {d_{\! H}\! \left(\zeta,\xi\right)} 
~\leq~ \frac {\left(\alpha_f\right)^{\frac {2}{\nu}}} {C_{\epsilon}-\epsilon}\, \max_{|\xi|_{\! H} \leq \epsilon} f(\xi)^{-\frac {2}{\nu}}
~<~ \epsilon.
\end{align*}
Thus, by Schauder's fixed point theorem, there exists $\xi^{\ast}=\xi^{\ast}(\zeta)\in \overline{B_{\epsilon}(0)}$ such that 
$T(\xi^{\ast})=\xi^{\ast}$, which clearly implies \eqref{fixedpoint1}.
\end{proof}

\begin{lem}\label{lem:asymp}
Let $f$ be the function from Theorem~\ref{thm:lnzlem2}. Suppose for small $\delta>0$ the maps $z \colon (-\delta, \delta) \setminus \{0\} \to \mathbb{C}^n$ and $t \colon (-\delta, \delta) \setminus \{0\} \to \mathbb{R}$ are continuous and satisfy one of the following two asymptotic conditions:
\begin{enumerate}
\item[(A1)]\, $z(h) = O(1)$ and $|t(h)| \approx |h|^{-1}$;
\item[(A2)]\, $t(h) = O(1)$ and $|z(h)| \approx |h|^{-1}$.
\end{enumerate}
Here, $a(h) \approx b(h)$ means that the ratio $a(h)/b(h)$ is bounded above and below by positive constants as $h \to 0$.

Then 
\begin{equation}\label{eqn:asym2}
\lim_{h \to 0} \, \frac{1}{h} \Big\{ \bigl|(z(h), t(h))\bigr|_{\! H}^{\nu}  f(z(h), t(h)) - \alpha_f \Big\} = 0.
\end{equation}
\end{lem}

\begin{proof}
Suppose that Condition (A1) is satisfied. It is then straightforward to verify that for any $(z', t') \in \mathbb{H}^n$ and sufficiently small $|h|$, 
\[
\bigl|(z(h),t(h)) \bigr|_{\! H}^4 \sim |t(h)|^2,
\]
and
\[
d_{\! H} \!\left((z(h), t(h)), (z', t')\right)^4 - \bigl|(z(h),t(h)) \bigr|_{\! H}^4
= \bigl[ -2t' + 4 \operatorname{Im} \langle z(h), z' \rangle \bigr] t(h) + o(t(h)).
\]
Here and throughout, $a(h) \sim b(h)$ means that $\lim_{h\to 0} a(h)/b(h) =1$. 
This should cause no confusion with the notation $a(h) \approx b(h)$.
It follows that
\begin{align}
& \bigl|(z(h), t(h)) \bigr|_{\! H}^{\nu} f(z(h), t(h)) \notag \\
&\quad = \lambda(z', t')^{\nu} \left\{ \frac{\bigl|(z(h),t(h)) \bigr|_{\! H}^4}{d_{\! H} \!\left((z(h), t(h)), (z', t')\right)^4} \right\}^{\nu/4} 
f \!\left( \Phi_{(z', t'), \lambda(z', t')}^{\, \beta_f} (z(h), t(h)) \right) \notag \\
&\quad = \lambda(z', t')^{\nu} \left\{ 1 + \frac{\nu}{2} \cdot \frac{t' - 2 \operatorname{Im} \langle z(h), z' \rangle}{t(h)} + o(h) \right\}
f \!\left( \Phi_{(z', t'), \lambda(z', t')}^{\, \beta_f} (z(h), t(h)) \right). \notag
\end{align}
Together with \eqref{eqn:alphaf1}, this implies
\begin{align}
& \frac{1}{|h|} \Bigl\{ \bigl|(z(h),t(h)) \bigr|_{\! H}^{\nu} f(z(h), t(h)) - \alpha_f \Bigr\} \notag \\
&\quad = \lambda(z', t')^{\nu} \cdot \frac{ f \!\left( \Phi_{(z', t'), \lambda(z', t')}^{\, \beta_f}
(z(h), t(h)) \right) - f(z', t') }{|h|} \notag \\
&\qquad + \lambda(z', t')^{\nu} \left\{ \frac{\nu}{2} \cdot \frac{t' - 2 \operatorname{Im} \langle z(h), z' \rangle}{t(h) |h|} + o(1) \right\}
f \!\left( \Phi_{(z', t'), \lambda(z', t')}^{\, \beta_f} (z(h), t(h)) \right). \label{eqn:asym02}
\end{align}
Taking $(z', t') = (0, 0)$ in the above and noting that $f$ attains its maximum at the origin, we obtain
\begin{equation}
\limsup_{h \to 0} \, \frac{1}{|h|} \Bigl\{ \bigl|(z(h),t(h)) \bigr|_{\! H}^{\nu} f(z(h), t(h)) - \alpha_f 
\Bigr\} ~\leq~ 0.
\end{equation}

On the other hand, Condition (A1) implies that there exist constants $c_1, c_2 > 0$ such that $|z(h)| \leq c_1$ and $|t(h) h| \geq c_2$ for all $h \in (-\delta, \delta) \setminus \{0\}$. Hence, for any $\epsilon \in (0, 1)$,
\[
\max_{|(z', t')|_{\! H} \leq \epsilon} \left| \frac{t' - 2 \operatorname{Im} \langle z(h), z' \rangle}{t(h) |h|} \right|
\leq \frac{1 + 2c_1}{c_2} \, \epsilon.
\]
Now, letting $(z', t')$ be the point $\xi^*$ from Lemma~\ref{lem:fixedpoint1}, it follows from \eqref{eqn:asym02} that
\begin{equation*}
\liminf_{h \to 0} \, \frac{1}{|h|} \Bigl\{ \bigl|(z(h),t(h)) \bigr|_{\! H}^{\nu} f(z(h), t(h)) - \alpha_f \Bigr\}
\geq - \epsilon \cdot \frac{\nu (1 + 2c_1) \alpha_f}{2 c_2} f(0) \max_{|\xi|_{\! H} \leq \epsilon} f(\xi)^{-1}.
\end{equation*}
Letting $\epsilon \to 0$, we conclude that
\begin{equation}
\liminf_{h \to 0} \, \frac{1}{|h|} \Bigl\{ \bigl|(z(h),t(h)) \bigr|_{\! H}^{\nu} f(z(h), t(h)) 
- \alpha_f \Bigr\} ~\geq~ 0.
\end{equation}
Therefore, \eqref{eqn:asym2} is established.

We now assume that Condition (A2) holds. In this case, for any $(z',t') \in \mathbb{H}^{n}$
and small $|h|$, we have
\[
\bigl|(z(h),t(h)) \bigr|_{\! H} \sim |z(h)|
\]
and
\[
d_{\! H}  \bigl((z(h), t(h)), (z', t')\bigr)^4 - \bigl|(z(h), t(h)\bigr|_{\! H}^4
= - 4 |z(h)|^2 \operatorname{Re} \langle z(h), z'\rangle 
+ o(|z(h)|^3).
\]
It follows that
\begin{align}
& \bigl|(z(h),t(h)) \bigr|_{\! H}^{\nu} f(z(h),t(h)) \notag\\
&\quad =~ \lambda(z',t')^{\nu} \left\{ \frac {\bigl|(z(h),t(h)) \bigr|_{\! H}^4} 
{d_{\! H} \! \left((z(h), t(h)), \left(z', t'\right)\right)^4}\right\}^{\nu/4} 
f\left(\Phi_{(z',t'),\lambda(z',t')}^{\, \beta_f} (z(h),t(h))\right) \notag\\
&\quad =~ \lambda(z',t')^{\nu} \left\{ 1+ \nu\, 
\frac {\operatorname{Re} \langle z(h), z'\rangle} {|z(h)|^2}
+ o \left( h \right) \right\} 
f\left(\Phi_{(z',t'),\lambda(z',t')}^{\, \beta_f} (z(h),t(h))\right), \notag
\end{align}
hence
\begin{align}
&\frac {1}{|h|} \, \big\{ \bigl|(z(h),t(h)) \bigr|_{\! H}^{\nu} f(z(h),t(h)) - \alpha_f \big\}\notag\\
& \quad \; ~=~ \lambda(z',t')^{\nu} \, \frac {f\left(\Phi_{(z', t'), \lambda(z',t')}^{\, \beta_f}
(z(h),t(h))\right) -f(z',t')}{|h|} \notag\\
& \qquad \quad + \lambda(z',t')^{\nu} \left\{  \nu\, 
\frac {\operatorname{Re} \langle z(h), z'\rangle} {|z(h)|^2 |h|}  + o(1) \right\} \, f\left(\Phi_{(z',t'),\lambda(z',t')}^{\, \beta_f} (z(h),t(h))\right).  \label{eqn:asym01}
\end{align}
Again, taking $(z',t')=(0,0)$ in the above and using the fact that $f$ has a maximum point at the origin, we obtain
\begin{equation}\label{eqn:limsup2}
\limsup_{h\to 0}\, \frac {1}{|h|} \, \Big\{ \bigl|(z(h),t(h))\bigr|_{\! H}^{\nu} f(z(h),t(h)) - \alpha_f \Big\} ~\leq~ 0.
\end{equation}

On the other hand, Condition (A2) implies that there exists a constant $c'_2>0$ such 
that $|z(h)||h|\geq c_2'$ for all $h\in (-\delta,\delta)\setminus \{0\}$. Thus,
for any $\epsilon\in (0,1)$,  
\[
\max_{|(z',t')|_{\! H} \leq \epsilon} \left| \frac {\operatorname{Re}\, 
\langle z(h), z'\rangle}  {|z(h)|^2 |h|} \right| ~\leq~  \frac {1}{c_2'} \epsilon. 
\]
Letting $(z',t')$ be the point $\xi^{\ast}$ from Lemma \ref{lem:fixedpoint1}, it follows from \eqref{eqn:asym01} that
\begin{equation*}
\liminf_{h\to 0}\, \frac {1}{|h|} \, \Big\{ \bigl|(z(h),t(h)) \bigr|_{\! H}^{\nu} f(z(h),t(h)) - \alpha_f \Big\} 
~\geq~ - \epsilon\, \frac {\nu \alpha_f}{c_2'}  f(0) \max_{|\xi|_{\! H} \leq \epsilon} f(\xi)^{-1}.
\end{equation*}
Sending $\epsilon$ to 0 , we have
\begin{equation}
\liminf_{h\to 0}  \frac {1}{|h|} \, \Big\{ \bigl|(z(h),t(h)) \bigr|_{\! H}^{\nu} f(z(h),t(h)) - \alpha_f \Big\}
~\geq~ 0.
\end{equation}
Together with \eqref{eqn:limsup2}, this completes the proof of the lemma.
\end{proof}

\section{Proof of Theorem \ref{thm:lnzlem2}}
\label{sec:PfofLNZ2}

\subsection*{Step 1}
The function $f$ admits the representation
\begin{equation}\label{eqn:oftheform1}
f(z,t) = \left(F(z)^2 + (\alpha_f)^{-\frac{4}{\nu}} t^{2}\right)^{-\frac{\nu}{4}},
\end{equation}
for some function $F$ defined on $\mathbb{C}^n$.

\begin{proof}
Fix $\xi = (z', t') \in \mathbb{H}^n$. For $h \in \mathbb{R}$, define
\[
(z(h), t(h)) ~:=~ \Phi_{(z',t'),\lambda(z',t')}^{\, \beta_f} (z', t' + h).
\]
A straightforward computation yields
\[
z(h) = z' \quad \text{and} \quad t(h) = t' + \frac{\lambda(z',t')^4}{h}.
\]
From \eqref{eqn:asym02} and \eqref{eqn:asym2}, we obtain
\begin{align}
& \lim_{h \to 0} \frac{f\left(\Phi_{(z', t'), \lambda(z',t')}^{\, \beta_f} (z(h),t(h))\right) 
- f(z',t')}{h} \notag\\
& \qquad \quad ~=~ -\lim_{h \to 0} \left\{ \frac{\nu}{2} \cdot \frac{t' - 2 \operatorname{Im} \langle z(h), z'\rangle}{t(h) h} + o(1) \right\} 
f\left(\Phi_{(z',t'),\lambda(z',t')}^{\, \beta_f} (z(h),t(h))\right) \notag\\
& \qquad \quad ~=~ -\frac{\nu}{2} \cdot t' \lambda(z',t')^{-4} f(z',t'). \label{eqn:asym03}
\end{align}
Since $\Phi_{(z',t'),\lambda(z',t')}^{\, \beta_f}$ is an involution,
\[
\Phi_{(z',t'),\lambda(z',t')}^{\, \beta_f} \left(z(h), t(h)\right) = (z', t' + h),
\]
so \eqref{eqn:asym03} becomes
\[
\lim_{h \to 0} \frac{f(z', t' + h) - f(z', t')}{h} 
= -\frac{\nu}{2} \cdot t' \lambda(z',t')^{-4} f(z', t'),
\]
or equivalently,
\[
\frac{\partial f}{\partial t}(z', t')
= -\frac{\nu}{2} \cdot (\alpha_f)^{-\frac{4}{\nu}} t' f(z', t')^{1 + \frac{4}{\nu}}.
\]
Solving this differential equation establishes the representation stated in Step 1.
\end{proof}

\subsection*{Step 2}
For \( k = 1, \dots, n \), we have
\begin{equation}\label{eqn:partialx}
\frac{\partial f}{\partial x_k}(z,0) = -\nu (\alpha_f)^{-\frac{2}{\nu}} x_k f(z,0)^{1+\frac{2}{\nu}} .
\end{equation}

\begin{proof}
Let \( e_1 = (1,0,\dots,0), \dots, e_n = (0,\dots,0,1) \) be the standard basis of \( \mathbb{C}^n \).  
Fix \( \xi = (z',0) \in \mathbb{H}^n \). For \( k \in \{1,\dots,n\} \) and \( h \in \mathbb{R} \), define  
\[
(z(h), t(h)) \;:=\; \Phi_{(z',0),\,\lambda(z',0)}^{\, \beta_f} \bigl(z' + h e_k,\; 0\bigr).
\]
Direct computation gives
\begin{align}
z(h) &~=~ z' + \frac{i \lambda(z',0)^2 e^{2i\arg z_k'}}{2 \operatorname{Im} z_k' + i h}\, e_k, \\[2mm]
t(h) &~=~ \frac{-2 \lambda(z',0)^4 \operatorname{Im} z_k'}{h^3 + 4h (\operatorname{Im} z_k')^2}
       + 2 \lambda(z',0)^2 \operatorname{Im} \Bigl(\frac{-i \overline{z_k'}}{2 \operatorname{Im} z_k' - i h}\Bigr).
\end{align}

\subsubsection*{Case I: \( \operatorname{Im} z_k' = 0 \)}
In this case
\[
z(h) = z' + \frac{\lambda(z',0)^2 e^{2i\arg z_k'}}{h}\, e_k ,\qquad t(h) = 0 .
\]
Hence \( |z(h)| \sim \lambda(z',0)^2 |h|^{-1} \) and
\[
\operatorname{Re}\langle z(h), z'\rangle
~=~ |z'|^2 + \frac{\lambda(z',0)^2}{h}\, \operatorname{Re} \bigl(e^{2i\arg z_k'}\overline{z}_k'\bigr)
~=~ |z'|^2 + \frac{\lambda(z',0)^2}{h}\, \operatorname{Re} z_k'.
\]
Using \eqref{eqn:asym01} and \eqref{eqn:asym2} we obtain
\begin{align*}
&\lim_{h\to 0}\frac{f \bigl(\Phi_{(z',0),\lambda(z',0)}^{\, \beta_f} (z(h),t(h))\bigr)-f(z',0)}{h} \\
&\quad =~ -\lim_{h\to0}\Bigl\{\nu\cdot\frac{\operatorname{Re}\langle z(h),z'\rangle}{|z(h)|^{2}\,h}+o(1)\Bigr\}
        f \bigl(\Phi_{(z',0),\lambda(z',0)}^{\, \beta_f} (z(h),t(h))\bigr) \\
&\quad =~ -\nu\,(\operatorname{Re} z_k')\,\lambda(z',0)^{-2}\,f(z',0).
\end{align*}
Noticing $\Phi_{(z',0),\lambda(z',0)}^{\, \beta_f} (z(h),t(h))=(z'+h e_k,0)$, we get
\[
\lim_{h\to0}\frac{f(z'+h e_k,0)-f(z',0)}{h}
~=~ -\nu\,(\operatorname{Re} z_k')\,\lambda(z',0)^{-2}\,f(z',0),
\]
which is precisely the assertion of Step 2.

\subsubsection*{Case II: \( \operatorname{Im} z_k' \neq 0 \)}
Now we have
\[
z(h)=O(1),\qquad t(h)\sim -\frac{\lambda(z',0)^4}{2h\operatorname{Im} z_k'},
\]
and
\[
\operatorname{Im}\langle z(h),z'\rangle
~=~ \lambda(z',0)^2\operatorname{Im} \Bigl(\frac{i e^{2i\arg z_k'}\overline{z}_k'}
{2\operatorname{Im} z_k'+i h}\Bigr)
~=~ \frac{\lambda(z',0)^2\operatorname{Re} z_k'}{2\operatorname{Im} z_k'}+o(1).
\]
Consequently,
\begin{align*}
&\lim_{h\to 0}\frac{f\!\bigl(\Phi_{(z',0),\lambda(z',0)}^{\, \beta_f} (z(h),t(h))\bigr)-f(z',0)}{h} \\
&\quad ~=~ \lim_{h\to0}\Bigl\{\frac{\nu\,\operatorname{Im}\langle z(h),z'\rangle}{t(h)\,h}+o(1)\Bigr\}
        f \bigl(\Phi_{(z',0),\lambda(z',0)}^{\, \beta_f} (z(h),t(h))\bigr) \\
&\quad ~=~ -\nu\,(\operatorname{Re} z_k')\,\lambda(z',0)^{-2}\,f(z',0).
\end{align*}
Again this yields
\[
\frac{\partial f}{\partial x_k}(z',0)=-\nu\,(\alpha_f)^{-\frac{2}{\nu}}x_k'\,f(z',0)^{1+\frac{2}{\nu}} .
\]

Both cases together complete the proof of Step 2.
\end{proof}

\subsection*{Step 3}
For \( k = 1, \dots, n \),
\begin{equation}\label{eqn:partialy}
\frac{\partial f}{\partial y_k}(z,0)= -\nu\,(\alpha_f)^{-\frac{2}{\nu}}\,y_k\,f(z,0)^{1+\frac{2}{\nu}} .
\end{equation}

\begin{proof}
Fix \( \xi=(z',0)\in\mathbb{H}^n \). For \( k\in\{1,\dots,n\} \) and \( h\in\mathbb{R} \) set
\[
(z(h),t(h)) ~:=~ \Phi_{(z',0),\,\lambda(z',0)}^{\, \beta_f} \bigl(z' + i h e_k,\; 0\bigr).
\]
Direct computation gives
\begin{align*}
z(h) &~=~ z' + \frac{-i\,\lambda(z',0)^2 e^{2i\arg z_k'}}{-2\operatorname{Re} z_k' + i h}\,e_k, \\[2mm]
t(h) &~=~ \frac{2\lambda(z',0)^4\operatorname{Re} z_k'}{h^{3}+4h(\operatorname{Re} z_k')^{2}}
       - 2\lambda(z',0)^2\operatorname{Im} \Bigl(\frac{i e^{-2i\arg z_k'}z_k'}
       {2\operatorname{Re} z_k' + i h}\Bigr).
\end{align*}

\subsubsection*{Case I: \( \operatorname{Re} z_k' = 0 \)}
Here
\[
z(h)=z'-\frac{\lambda(z',0)^2 e^{2i\arg z_k'}}{h}\,e_k,\qquad t(h)=0 .
\]
Thus \( |z(h)|\sim\lambda(z',0)^2|h|^{-1} \) and
\[
\operatorname{Re}\langle z(h),z'\rangle = |z'|^2 + \frac{\lambda(z',0)^2}{h}\,\operatorname{Im} z_k'.
\]
From \eqref{eqn:asym01} and \eqref{eqn:asym2} we obtain
\begin{align*}
&\lim_{h\to0}\frac{f(z'+ i h e_k,0)-f(z',0)}{h} \\
&\quad ~=~ -\lim_{h\to0}\Bigl\{\nu\cdot\frac{\operatorname{Re}\langle z(h),z'\rangle}{|z(h)|^{2}\,h}+o(1)\Bigr\}
        f \bigl(\Phi_{(z',0),\lambda(z',0)}^{\, \beta_f} (z(h),t(h))\bigr) \\
&\quad ~=~ -\nu\,(\operatorname{Im} z_k')\,\lambda(z',0)^{-2}\,f(z',0).
\end{align*}

\subsubsection*{Case II: \( \operatorname{Re} z_k' \neq 0 \)}
Now
\[
z(h)=O(1),\qquad t(h)\sim -\frac{\lambda(z',0)^4}{2h\operatorname{Re} z_k'},
\]
and
\[
\operatorname{Im}\langle z(h),z'\rangle
~=~ -\frac{\lambda(z',0)^2\operatorname{Im} z_k'}{2\operatorname{Re} z_k'} + o(1).
\]
So we have
\begin{align*}
&\lim_{h\to0}\frac{f(z'+ i h e_k,0)-f(z',0)}{h} \\
&\quad ~=~ \lim_{h\to0}\Bigl\{\frac{\nu\,\operatorname{Im}\langle z(h),z'\rangle}{t(h)\,h}+o(1)\Bigr\}
        f \bigl(\Phi_{(z',0),\lambda(z',0)}^{\, \beta_f} (z(h),t(h))\bigr) \\
&\quad ~=~ -\nu\,(\operatorname{Im} z_k')\,\lambda(z',0)^{-2}\,f(z',0).
\end{align*}

In both cases we conclude that
\[
\frac{\partial f}{\partial y_k}(z',0)= -\nu\,(\alpha_f)^{-\frac{2}{\nu}}\,y_k'\,f(z',0)^{1+\frac{2}{\nu}},
\]
which is exactly the assertion of Step 3.
\end{proof}

\subsection*{Step 4}

Now we recall that $f(z,0)=F(z)^{-\nu/2}$. Thus Steps 2 and 3 imply that
\[
\frac {\partial F}{\partial x_k} = 2 (\alpha_f)^{-\frac {2}{\nu}} x_k \quad \text{and} \quad 
\frac {\partial F}{\partial y_k} = 2 (\alpha_f)^{-\frac {2}{\nu}} y_k
\] 
Hence $F(z)=(\alpha_f)^{-\frac {2}{\nu}} |z|^2 + C$ for some constant $C$. Since $F(0)=f(0,0)^{-\frac {2}{\nu}}$,
we find that $C=f(0,0)^{-\frac {2}{\nu}} = (\alpha_f)^{-\frac {2}{\nu}} \beta_f$. Therefore,
\[
f(z,t) = \alpha_f \Big[(|z|^2 + \beta_f)^2 + t^2\Big]^{-\frac {\nu}{4}}
= \alpha_f \big|t + i|z|^2 + i \beta_f \big|^{-\frac {\nu}{2}}.
\]
Theorem \ref{thm:lnzlem2} is established.

\section{Terracini-type integral inequality}
\label{sec:Terracini}

The proof of Theorem \ref{thm:main} employs the method of moving spheres combined with an integral inequality
(in place of the maximum principle). This idea is originally due to S. Terracini \cite{Terracini95, Terracini96}.
Therefore, we refer to the following lemma as the Terracini-type integral inequality.

\begin{lem}\label{lem:terracini}
Let $u$ be a solution of the equation
\begin{equation*}
-\Delta_{\! H} u = u^{p},
\end{equation*}
where $1<p\leq \frac {Q+2}{Q-2}$, and let
\[
A_{\lambda}(\xi) := \left\{ \zeta \in B_{\lambda}(\xi) : u(\zeta) > u_{\xi,\lambda}^{\,\beta}(\zeta) \right\}.
\]
Then, for any $\lambda > 0$, there exists a constant $C > 0$, independent of $\lambda$, such that
\begin{equation}\label{eqn:terracini}
\int_{B_{\lambda}(\xi)} 
\left|\nabla_{\! H} (u - u_{\xi,\lambda}^{\,\beta})^{+}\right|^2 
\leq C \left\{ \int_{A_{\lambda}(\xi)} u^{\frac{(p-1)Q}{2}} \right\}^{\frac{2}{Q}} 
\int_{B_{\lambda}(\xi)} \left|\nabla_{\! H} (u - u_{\xi,\lambda}^{\,\beta})^{+}\right|^2.
\end{equation}
\end{lem}

\begin{proof}
Given $\epsilon > 0$, pick a real cut-off function $g_{\epsilon} \in C_{c}^{\infty}(\mathbb{H}^n \setminus \{0\})$ such that:
\begin{itemize}
  \item $0 \leq g_{\epsilon} \leq 1$;
  \item $\operatorname{supp} g_{\epsilon} \subset \left\{ \zeta \in \mathbb{H}^n : \epsilon \leq |\zeta|_{\! H} \leq 2\epsilon^{-1} \right\}$;
  \item $g_{\epsilon} \equiv 1$ on $\left\{ \zeta \in \mathbb{H}^n : 2\epsilon \leq |\zeta|_{\! H} \leq \epsilon^{-1} \right\}$;
  \item $\left|\nabla_{\! H} g_{\epsilon}(\zeta)\right| \leq C\epsilon^{-1}$ for $\epsilon < |\zeta|_{\! H} < 2\epsilon$;  
  \item $\left|\nabla_{\! H} g_{\epsilon}(\zeta)\right| \leq C\epsilon$ for $\epsilon^{-1} < |\zeta|_{\! H} < 2\epsilon^{-1}$.
\end{itemize}
Define
\[
\eta_{\epsilon}(\zeta) := g_{\epsilon}(\xi^{-1} \cdot \zeta), \quad \zeta \in \mathbb{H}^n \setminus \{\xi\}
\]
and
\begin{align}\label{eqn:phiepsilon}
\psi_{\epsilon} &:= (u - u_{\xi,\lambda}^{\,\beta})^{+} \eta_{\epsilon},  \notag\\
\phi_{\epsilon} &:= (u - u_{\xi,\lambda}^{\,\beta})^{+} \eta_{\epsilon}^2.
\end{align}
Direct calculation shows
\begin{align*}
- \int_{A_{\lambda}(\xi)} \phi_{\epsilon} \Delta_{\! H} (u - u_{\xi,\lambda}^{\,\beta}) 
&= \int_{B_{\lambda}(\xi)} \nabla_{\! H} (u - u_{\xi,\lambda}^{\,\beta})^{+} \cdot \nabla_{\! H} \phi_{\epsilon} \\
&= \int_{B_{\lambda}(\xi)} \left| \nabla_{\! H} \psi_{\epsilon} \right|^2 - \int_{B_{\lambda}(\xi)} \left| (u - u_{\xi,\lambda}^{\,\beta})^{+} \right|^2 \left| \nabla_{\! H} \eta_{\epsilon} \right|^2.
\end{align*}
On the other hand, by Corollary \ref{cor:conformal_invariance} and the mean value theorem,
we have
\begin{align*}
- \int_{A_{\lambda}(\xi)} \phi_{\epsilon} \Delta_{\! H} (u - u_{\xi,\lambda}^{\,\beta}) 
&= \int_{A_{\lambda}(\xi)} \phi_{\epsilon} \left[ u^{p} - \left( \frac{\lambda}{ d_{\! H} 
(\xi, \,\cdot\,) } \right)^{(Q+2) - (Q-2)p} \left(u_{\xi,\lambda}^{\,\beta}\right)^{p} \right] \\
&\leq p \int_{A_{\lambda}(\xi)} u^{p-1} (u - u_{\xi,\lambda}^{\,\beta}) \phi_{\epsilon} \\
&= p \int_{A_{\lambda}(\xi)} u^{p-1} (u - u_{\xi,\lambda}^{\,\beta})^2 \eta_{\epsilon}^2.
\end{align*}
Then, by H\"older's inequality,
\begin{align*}
- \int_{A_{\lambda}(\xi)} \phi_{\epsilon} \Delta_{\! H} (u - u_{\xi,\lambda}^{\,\beta}) 
&\leq p \left\{ \int_{A_{\lambda}(\xi)} u^{\frac{(p-1)Q}{2}} \right\}^{\frac{2}{Q}} 
\left\{ \int_{A_{\lambda}(\xi)} \left[ (u - u_{\xi,\lambda}^{\,\beta})^2 \eta_{\epsilon}^2 \right]^{\frac{Q}{Q-2}} \right\}^{\frac{Q-2}{Q}}.
\end{align*}
But, by the Sobolev inequality on the Heisenberg group (see \cite{FS74})
\begin{equation}
 \label{sobol1}
\|f\|_{\frac{2Q}{Q-2}} \leq C \|\nabla_{\! H}f\|_2
\end{equation}
we have
\[
\left\{ \int_{A_{\lambda}(\xi)} \left[ (u - u_{\xi,\lambda}^{\,\beta})^2 \eta_{\epsilon}^2 \right]^{\frac{Q}{Q-2}} \right\}^{\frac{Q-2}{Q}}
\leq C \int_{B_{\lambda}(\xi)} \left| \nabla_{\! H} (u - u_{\xi,\lambda}^{\,\beta})^{+} \right|^2.
\]
Thus
\begin{align*}
\int_{B_{\lambda}(\xi)} \left| \nabla_{\! H} \psi_{\epsilon} \right|^2 
&\leq C \left\{ \int_{A_{\lambda}(\xi)} u^{\frac{(p-1)Q}{2}} \right\}^{\frac{2}{Q}} 
\int_{B_{\lambda}(\xi)} \left| \nabla_{\! H} (u - u_{\xi,\lambda}^{\,\beta})^{+} \right|^2 \\
&\quad + \int_{B_{\lambda}(\xi)} \left| (u - u_{\xi,\lambda}^{\,\beta})^{+} \right|^2 \left| \nabla_{\! H} \eta_{\epsilon} \right|^2.
\end{align*}

Since $\psi_{\epsilon} = (u - u_{\xi,\lambda}^{\,\beta})^{+} \eta_{\epsilon}$, we have
\[
\int_{B_{\lambda}(\xi)} \left| \nabla_{\! H} \psi_{\epsilon} \right|^2 
\geq \int_{B_{\lambda}(\xi) \cap \{2\epsilon \leq d_{\! H}(\zeta,\xi) \leq \epsilon^{-1}\}} 
\left| \nabla_{\! H} (u - u_{\xi,\lambda}^{\,\beta})^{+} \right|^2.
\]

Hence
\begin{align}\label{eqn:nablaest}
\int_{B_{\lambda}(\xi) \cap \{2\epsilon \leq d_{\! H}(\zeta,\xi) \leq \epsilon^{-1}\}} 
& \left| \nabla_{\! H} (u - u_{\xi,\lambda}^{\,\beta})^{+} \right|^2 \notag \\
& \leq C \left\{ \int_{A_{\lambda}(\xi)} u^{\frac{(p-1)Q}{2}} \right\}^{\frac{2}{Q}} 
\int_{B_{\lambda}(\xi)} \left| \nabla_{\! H} (u - u_{\xi,\lambda}^{\,\beta})^{+} \right|^2 \notag \\
& \quad + \underbrace{\int_{B_{\lambda}(\xi)} \left| (u - u_{\xi,\lambda}^{\,\beta})^{+} \right|^2 \left| \nabla_{\! H} \eta_{\epsilon} \right|^2}_{I(\epsilon)}.
\end{align}

Define
\[
\Omega_{\epsilon} = \{\zeta \in B_{\lambda}(\xi) : \epsilon < d_{\! H}(\zeta,\xi) < 2\epsilon\}
\cup \{\zeta \in B_{\lambda}(\xi) : \epsilon^{-1} < d_{\! H}(\zeta,\xi) < 2\epsilon^{-1}\}.
\] 

Note that $\Omega_{\epsilon} \to \emptyset$ as $\epsilon \to 0$, and
\[
\left| \nabla_{\! H} \eta_{\epsilon} \right|^Q |\Omega_{\epsilon}| \leq C\left[(\epsilon^{-1})^Q \epsilon^Q + \epsilon^Q (\epsilon^{-1})^Q\right] = C.
\]

Therefore
\begin{align*}
I(\epsilon) &\leq \left\{ \int_{\Omega_{\epsilon}} \left[ (u - u_{\xi,\lambda}^{\,\beta})^{+} \right]^{\frac{2Q}{Q-2}} \right\}^{\frac{Q-2}{Q}} 
\left\{ \int_{\Omega_{\epsilon}} \left| \nabla_{\! H} \eta_{\epsilon} \right|^Q \right\}^{\frac{2}{Q}} \\
&\leq C \left\{ \int_{\Omega_{\epsilon}} \left[ (u - u_{\xi,\lambda}^{\,\beta})^{+} \right]^{\frac{2Q}{Q-2}} \right\}^{\frac{Q-2}{Q}} \to 0 \qquad \text{as }\, \epsilon\to 0.
\end{align*}

Taking the limit as $\epsilon \to 0$ in \eqref{eqn:nablaest} yields \eqref{eqn:terracini}.
\end{proof}

\section{The Proof of Theorem \ref{thm:main}}
\label{sec:PfofMainThm}

\subsection*{Step 1}
We compare $u$ with its generalized Kelvin transform $u_{\xi,\lambda}^{\,\beta}$, with $\beta$ 
to be determined later.
We first show that, for any $\xi \in \mathbb{H}^{n}$, there exists $\lambda_{0}(\xi)>0$ such that
\begin{equation}\label{eqn:ulessthankelvin}
u(\zeta) \leq u_{\xi,\lambda}^{\,\beta}(\zeta)
\end{equation}
for all $\zeta \in B_{\lambda}(\xi)$ and all  $\lambda \in\left(0, \lambda_{0}(\xi)\right)$.

Define
\[
A_{\lambda}(\xi):=\left\{\zeta \in B_{\lambda}(\xi): u(\zeta) > u_{\xi,\lambda}^{\,\beta}(\zeta)\right\}.
\]
This is the set where inequality \eqref{eqn:ulessthankelvin} is violated. We show that
for sufficiently small $\lambda$,  $A_{\lambda}(\xi)$ must be empty. 

Since $u\in C^2(\mathbb{H}^n)$,
\[
\lim_{\lambda\to 0} \int\limits_{A_{\lambda}(\xi)} u^{\frac {2Q}{Q-2}} 
=0,
\]
so, there exists a sufficiently small $\lambda_0$ such that
\begin{align*}
C \Bigg\{ \int_{A_{\lambda}(\xi)} u^{\frac {2Q}{Q-2}} \Bigg\}^{\frac {2}{Q}} < \frac {1}{2}
\end{align*}
for all $\lambda\in (0,\lambda_0)$.
Then, in view of the Terracini-type inequality \eqref{eqn:terracini}, for these $\lambda$, we have 
\[
\int\limits_{B_{\lambda}(\xi)} \left|\nabla_{\! H} (u-u_{\xi,\lambda}^{\,\beta})^{+}\right|^2 =0.
\]
Therefore, with $\phi_{\epsilon}$ defined as in \eqref{eqn:phiepsilon}, we have
\begin{align*}
0 ~=~& \int\limits_{B_{\lambda}(\xi)}  \nabla_{\! H} (u-u_{\xi,\lambda}^{\,\beta})^{+} \cdot \nabla_{\! H} \phi_{\epsilon}\\
=~& - \int\limits_{A_{\lambda}(\xi)} \phi_{\epsilon} \Delta_{\! H} (u-u_{\xi,\lambda}^{\,\beta}) \\
=~& \int\limits_{A_{\lambda}(\xi)} \eta_{\epsilon}^2\cdot (u-u_{\xi,\lambda}^{\,\beta}) \left[u^{\frac {Q+2}{Q-2}} -(u_{\xi,\lambda}^{\,\beta})^{\frac {Q+2}{Q-2}}\right],
\end{align*}
which implies that $A_{\lambda}(\xi)$ must be measure zero and hence empty by the continuity of 
$u-u_{\xi,\lambda}^{\,\beta}$. This verifies \eqref{eqn:ulessthankelvin} and hence completes Step 1.

\subsection*{Step 2}

Define
\[
\underline{\lambda}(\xi) := \sup\left\{\mu > 0 : u(\zeta) \leq u_{\xi,\lambda}^{\,\beta}(\zeta),\ \forall\zeta\in B_{\lambda}(\xi),\ \forall \lambda\in (0, \mu)\right\}.
\]

\subsubsection*{Step 2.1}
There exists some $\xi_0 \in \mathbb{H}^n$ such that $\underline{\lambda}(\xi_0) < \infty$.

\begin{proof}
Suppose, for contradiction, that $\underline{\lambda}(\xi) = \infty$ for all $\xi \in \mathbb{H}^n$. By the definition of $\underline{\lambda}(\xi)$, this implies that for any $\xi \in \mathbb{H}^n$ and any $\lambda > 0$,
\[
u(\zeta) \leq u_{\xi,\lambda}^{\,\beta}(\zeta) \quad \text{for all } \zeta \in B_{\lambda}(\xi).
\]
By writing $\eta=\Phi_{\xi,\lambda}^{\, \beta}(\zeta)$, we conclude that for any $\xi \in \mathbb{H}^n$ and any $\lambda > 0$,
\[
u(\eta) \geq u_{\xi,\lambda}^{\,\beta}(\eta) \quad \text{for all } \eta \in \Sigma_{\lambda}(\xi).
\]
In view of our Calculus Lemma I (Theorem \ref{thm:lnzlem1}), $u$ must be constant. But the only constant solution to equation (1.1) is $u \equiv 0$, which leads to a contradiction.
\end{proof}

\subsubsection*{Step 2.2}
With $\xi_0$ as in Step 2.1, we have
\[
u_{\xi_0,\underline{\lambda}(\xi_0)} \equiv u \quad \text{on } \mathbb{H}^{n}.
\]

\begin{proof}
By symmetry, it suffices to prove that
\[
u_{\xi_0,\underline{\lambda}(\xi_0)}^{\,\beta}(\zeta) = u(\zeta) \quad \text{for all } \zeta \in B_{\underline{\lambda}(\xi_0)}(\xi_0).
\]

By the definition of $\underline{\lambda}(\xi_0)$, 
\[
u(\zeta) \leq u_{\xi_0,\underline{\lambda}(\xi_0)}^{\,\beta} (\zeta)  \quad \text{for all } \zeta\in  B_{\underline{\lambda}(\xi_0)}(\xi_0).
\]
On the other hand, by writing $\eta=\Phi_{\xi,\lambda}^{\, \beta}(\zeta)$, we have
\[
u(\eta) \geq u_{\xi_0,\underline{\lambda}(\xi_0)}^{\,\beta} (\eta) \quad \text{for all } \eta \in \Sigma_{\underline{\lambda}(\xi_0)}(\xi_0).
\]
Then, by the continuity $u_{\xi_0,\underline{\lambda}(\xi_0)}^{\,\beta}-u$,
\[
u_{\xi_0,\underline{\lambda}(\xi_0)}^{\,\beta} - u \equiv 0 \quad \text{on } \partial B_{\underline{\lambda}(\xi_0)}(\xi_0).
\]
Note also that
\[
-\Delta_{\! H} \left(u_{\xi_0,\underline{\lambda}(\xi_0)}^{\,\beta} - u\right) = \left(u_{\xi_0,\underline{\lambda}(\xi_0)}^{\,\beta}\right)^{\frac {Q+2}{Q-2}} - u^{\frac {Q+2}{Q-2}} \geq 0 \quad \text{on } B_{\underline{\lambda}(\xi_0)}(\xi_0).
\]
Thus, by the strong maximum principle (see, e.g., \cite{Bony69}), either
\[
u_{\xi_0,\underline{\lambda}(\xi_0)}^{\,\beta}(\zeta) = u(\zeta) \quad \text{for all } \zeta \in B_{\underline{\lambda}(\xi_0)}(\xi_0),
\]
in which case we are done, or
\[
u_{\xi_0,\underline{\lambda}(\xi_0)}^{\,\beta}(\zeta) > u(\zeta) \quad \text{for all } \zeta \in B_{\underline{\lambda}(\xi_0)}(\xi_0).
\]
We now show that the latter case leads to a contradiction.

We prove that the sphere can be moved further. Specifically, there exists $\epsilon > 0$, depending on $n$ and $u$, such that
\[
u_{\xi_0,\lambda}^{\,\beta} - u \geq 0 \quad \text{on } B_{\lambda}(\xi_0)
\]
for all $\lambda \in [\underline{\lambda}(\xi_0), \underline{\lambda}(\xi_0) + \epsilon)$.

Again, in view of the Terracini-type inequality \eqref{eqn:terracini}, it suffices to choose $\epsilon$ small enough so that
\begin{equation}\label{eqn:lessthanhalf2}
C \left\{ \int_{A_{\lambda}(\xi_0)} u^{\frac{2Q}{Q-2}} \right\}^{\frac{2}{Q}} < \frac{1}{2}
\end{equation}
holds for all $\lambda \in (\underline{\lambda}(\xi_0), \underline{\lambda}(\xi_0) + \epsilon)$.

Since $u_{\xi_0,\underline{\lambda}(\xi_0)}^{\,\beta}(\zeta) > u(\zeta)$ for all $\zeta \in B_{\underline{\lambda}(\xi_0)}(\xi_0)$, we define for any $\delta > 0$:
\[
E_{\delta} := \left\{\zeta \in B_{\underline{\lambda}(\xi_0)}(\xi_0) : u_{\xi_0,\underline{\lambda}(\xi_0)}^{\,\beta}(\zeta) - u(\zeta) > \delta\right\}, \quad
F_{\delta} := B_{\underline{\lambda}(\xi_0)}(\xi_0) \setminus E_{\delta}.
\]
Then clearly
\[
\lim_{\delta \rightarrow 0} |F_{\delta}| = 0.
\]
Note that
\begin{equation}\label{eqn:alambda1}
A_{\lambda}(\xi_0) \subset \left(A_{\lambda}(\xi_0) \cap E_{\delta}\right) \cup F_{\delta} \cup \left(B_{\lambda}(\xi_0) \setminus B_{\underline{\lambda}(\xi_0)}(\xi_0)\right).
\end{equation}

The measure of $B_{\lambda}(\xi_0) \setminus B_{\underline{\lambda}(\xi_0)}(\xi_0)$ becomes small as $\lambda$ approaches $\underline{\lambda}(\xi_0)$. We now show that the measure of $A_{\lambda}(\xi_0) \cap E_{\delta}$ can also be made arbitrarily small. For any $\zeta \in A_{\lambda}(\xi_0) \cap E_{\delta}$, we have
\[
0 < u(\zeta) - u_{\xi_0,\lambda}^{\,\beta}(\zeta) = u(\zeta) - u_{\xi_0,\underline{\lambda}(\xi_0)}^{\,\beta}(\zeta) + u_{\xi_0,\underline{\lambda}(\xi_0)}^{\,\beta}(\zeta) - u_{\xi_0,\lambda}^{\,\beta}(\zeta),
\]
so that
\[
u_{\xi_0,\underline{\lambda}(\xi_0)}^{\,\beta}(\zeta) - u_{\xi_0,\lambda}^{\,\beta}(\zeta) > u_{\xi_0,\underline{\lambda}(\xi_0)}^{\,\beta}(\zeta) - u(\zeta) > \delta.
\]
Hence,
\begin{equation}\label{eqn:alambda2}
A_{\lambda}(\xi_0) \cap E_{\delta} \subset G_{\delta} := \left\{\zeta \in B_{\lambda}(\xi_0) : u_{\xi_0,\underline{\lambda}(\xi_0)}^{\,\beta}(\zeta) - u_{\xi_0,\lambda}^{\,\beta}(\zeta) > \delta\right\}.
\end{equation}
By Chebyshev's inequality,
\[
|G_{\delta}| \leq \frac{1}{\delta} \int_{G_{\delta}} \left|u_{\xi_0,\underline{\lambda}(\xi_0)}^{\,\beta} - u_{\xi_0,\lambda}^{\,\beta}\right| \leq \frac{1}{\delta} \int_{B_{\lambda}(\xi_0)} \left|u_{\xi_0,\underline{\lambda}(\xi_0)}^{\,\beta} - u_{\xi_0,\lambda}^{\,\beta}\right|.
\]
For fixed $\delta$, the right-hand side can be made arbitrarily small as $\lambda \to \underline{\lambda}(\xi_0)$. Therefore, by \eqref{eqn:alambda1} and \eqref{eqn:alambda2}, the measure of $A_{\lambda}(\xi_0)$ can be made sufficiently small, and inequality \eqref{eqn:lessthanhalf2} follows.

We conclude that
\[
u_{\xi_0,\lambda}^{\,\beta}(\xi) \leq u(\xi) \quad \text{for all } \xi \in B_{\lambda}(\xi_0),
\]
for all $\lambda \in [\underline{\lambda}(\xi_0), \underline{\lambda}(\xi_0) + \varepsilon)$, contradicting the definition of $\underline{\lambda}(\xi_0)$.
\end{proof}

\subsubsection*{Step 2.3}
We have $\underline{\lambda}(\xi) < \infty$ for all $\xi \in \mathbb{H}^n$.

\begin{proof}
For any $\xi \in \mathbb{H}^n$, the definition of $\underline{\lambda}(\xi)$ implies that for all $\lambda \in (0, \underline{\lambda}(\xi))$,
\[
u_{\xi,\lambda}^{\,\beta}(\zeta) \leq u(\zeta) \quad \text{for all } \zeta \in \Sigma_{\lambda}(\xi).
\]
It follows that
\begin{equation}\label{eqn:takinglimit1}
\liminf_{|\zeta|_{\! H} \to \infty} \left(|\zeta|_{\! H}^{Q-2} u(\zeta)\right) \geq \liminf_{|\zeta|_{\! H} \to \infty} \left(|\zeta|_{\! H}^{Q-2} u_{\xi,\lambda}^{\,\beta}(\zeta)\right) = \lambda^{Q-2} u(\xi), \quad \forall \lambda \in (0, \underline{\lambda}(\xi)).
\end{equation}
On the other hand, by Step 2.2 we have
\begin{equation}\label{eqn:takinglimit2}
\liminf_{|\zeta|_{\! H} \to \infty} \left(|\zeta|_{\! H}^{Q-2} u(\zeta)\right) = \liminf_{|\zeta|_{\! H} \to \infty} \left(|\zeta|_{\! H}^{Q-2} u_{\xi_0,\underline{\lambda}(\xi_0)}(\zeta)\right) = \underline{\lambda}(\xi_0)^{Q-2} u(\xi_0) < \infty.
\end{equation}
Combining \eqref{eqn:takinglimit1} and \eqref{eqn:takinglimit2}, we obtain $\underline{\lambda}(\xi) < \infty$ for all $\xi \in \mathbb{H}^n$.
\end{proof}

\subsubsection*{Step 2.4}
We have the identity
\[
u_{\xi,\underline{\lambda}(\xi)}^{\,\beta} \equiv u \quad \text{on } \mathbb{H}^{n} \quad \text{for all } \xi \in \mathbb{H}^n.
\]

\begin{proof}
This follows by applying the argument in Step 2.2 to an arbitrary $\xi \in \mathbb{H}^n$.
\end{proof}

\subsection*{Step 3}

Given that $u(\zeta) \to 0$ as $|\zeta|_H \to \infty$ (see Remark \ref{rmk:lnzlem2}), and since $u$ is continuous and positive, $u$ must attain its maximum at some $\xi'=(z',t') \in \mathbb{H}^n$. Consider $U := u \circ \tau_{\xi'}$, which is also a solution of \eqref{eqn:cr_yamabe} and has its maximum at the origin. By applying Step 2.4 to $U$ with $\beta = \beta_U$ and using Theorem \ref{thm:lnzlem2}, it follows that
\[
U(z,t) = \alpha_U \left|t + i|z|^2 +i\beta_U\right|^{-\frac{Q-2}{2}}.
\]
Thus, we conclude that 
\begin{align*}
u(z,t) =& U \circ \tau_{(z',t')}^{-1} (z,t) \\
=& \alpha_U \left|t+i|z|^2-2i\langle z, z'\rangle - t' 
+ i\left|z' \right|^2 + i \beta_U \right|^{-\frac {Q-2}{2}}\\
=& \alpha_U \left|t+i|z|^2 + \langle z, \mu \rangle + \kappa \right|^{-\frac {Q-2}{2}}
\end{align*}
where
\[
\mu = 2iz'\quad \text{and} \quad
\kappa = -t' + i\left|z' \right|^2 + i \beta_U.
\]
This completes the proof.

\section{An alternative proof of Theorem D}
\label{sec:PfofMaOu}

Let $u$ be a nonnegative solution of \eqref{eqn:subcritical}. 
By an argument analogous to that in Step 1 of Section \ref{sec:PfofMainThm}, one can show that, for every $\xi \in \mathbb{H}^{n}$, there exists $\lambda_{0}(\xi)>0$ such that for all $\lambda \in (0, \lambda_{0}(\xi))$,
\[
u(\zeta) \geq u_{\xi,\lambda}^{\beta}(\zeta), \quad \forall \zeta \in \Sigma_{\lambda}(\xi).
\]

We then define
\[
\underline{\lambda}(\xi) := \sup \left\{ \mu > 0 : u(\zeta) \geq u_{\xi,\lambda}^{\beta}(\zeta),\ \forall \zeta \in \Sigma_{\lambda}(\xi),\ \forall \lambda \in (0, \mu) \right\}.
\]

\subsubsection*{Case I: $\underline{\lambda}(\xi) = \infty$ for all $\xi \in \mathbb{H}^n$}
By the definition of $\underline{\lambda}(\xi)$, this implies that for any $\xi \in \mathbb{H}^n$ and any $\lambda > 0$,
\[
u(\zeta) \geq u_{\xi,\lambda}^{\beta}(\zeta), \quad \forall \zeta \in \Sigma_{\lambda}(\xi).
\]
Using our Calculus Lemma I (Theorem \ref{thm:lnzlem1}), we conclude that $u$ is constant. However, the only constant solution of equation \eqref{eqn:subcritical} is $u \equiv 0$.

\subsubsection*{Case II: There exists some $\xi_0 \in \mathbb{H}^n$ such that $\underline{\lambda}(\xi_0) < \infty$}
The same argument from Step 2.2 in Section \ref{sec:PfofMainThm} shows that
\[
u_{\xi_0, \underline{\lambda}(\xi_0)}^{\, \beta} \equiv u \quad \text{on } \mathbb{H}^{n}\setminus \{\xi_0\}.
\]
Combined with Corollary \ref{cor:conformal_invariance}, this gives
\[
\begin{aligned}
0 & ~=~ -\Delta_{\! H} \bigl( u - u_{\xi_0, \underline{\lambda}(\xi_0)}^{\, \beta}  \bigr) \\
&=~ u^p - \left( \frac{ \underline{\lambda}(\xi_0) }{ d_{\! H} (\xi_0, \,\cdot\,) } \right)^{(Q+2) - (Q-2)p} \left(u_{\xi_0, \underline{\lambda}(\xi_0)}^{\, \beta}\right)^p \\
&=~ \left\{ 1 - \left( \frac{ \underline{\lambda}(\xi_0) }{ d_{\! H} (\xi_0, \,\cdot\,) } \right)^{(Q+2) - (Q-2)p} \right\} u^p,
\end{aligned}
\]
which implies that $u \equiv 0$ and completes the proof. 

\subsection*{Acknowledgments}
I am deeply grateful to Meijun Zhu for the many stimulating conversations we shared, from which I gained a deep understanding of the methods of moving planes and moving spheres.  I am also indebted to Guangbin Ren, Xieping Wang, Jungang Li, and Xinan Ma for their valuable comments and suggestions. Furthermore, I wish to thank Jingbo Dou for bringing reference \cite{HWZ17} to my attention.


\begin{thebibliography}{99}

\bibitem{BCC97}
I.~Birindelli, I.~Capuzzo-Dolcetta, and A.~Cutri,
\emph{Liouville theorems for semilinear equations on the Heisenberg group},
Ann. Inst. Henri Poincaré, Anal. Non Linéaire \textbf{14} (1997), 295--308.

\bibitem{BP99}
I.~Birindelli and J.~Prajapat,
\emph{Nonlinear Liouville theorems in the Heisenberg group via the moving plane},
Commun. Partial Differ. Equ. \textbf{24} (1999), 1875--1890.

\bibitem{Bony69}
J.-M.~Bony,
\emph{Principe du maximum, inégalité de Harnack et unicité du problème de Cauchy pour les opérateurs elliptiques dégénérés},
Ann. Inst. Fourier (Grenoble) \textbf{19} (1969), 277--304.

\bibitem{CGS89}
L.~A.~Caffarelli, B.~Gidas, and J.~Spruck,
\emph{Asymptotic symmetry and local behavior of semilinear elliptic equations with critical Sobolev growth},
Comm. Pure Appl. Math. \textbf{42} (1989), 271--297.

\bibitem{CLMR23}
G.~Catino, Y.~Li, D.~D.~Monticelli, and A.~Roncoroni,
\emph{A Liouville theorem in the Heisenberg group},
J. Eur. Math. Soc. (to appear),
arXiv preprint at arXiv:2310.10469.

\bibitem{FV23}
J.~Flynn and J.~V\'etois,
\emph{Liouville-type results for the CR Yamabe equation in the Heisenberg group},
Ann. Sc. Norm. Super. Pisa Cl. Sci. (to appear), 
arXiv preprint at arXiv:2310.14048.

\bibitem{FS74}
G.~B.~Folland and E.~M.~Stein,
\emph{Estimates for the $\bar{\partial}_b$ complex and analysis on the Heisenberg group},
Comm. Pure Appl. Math. \textbf{27} (1974), 429--522.

\bibitem{GV01}
N.~Garofalo and D.~Vassilev,
\emph{Symmetry properties of positive entire solutions of Yamabe-type equations on groups of Heisenberg type},
Duke Math. J. \textbf{106} (2001), 411--448.

\bibitem{GNN81}
B.~Gidas, W.~M.~Ni, and L.~Nirenberg,
\emph{Symmetry of positive solutions of nonlinear elliptic equations in $\mathbb{R}^n$},
Adv. Math. Suppl. Stud. \textbf{7a} (1981), 369--402.

\bibitem{HWZ17}
Y.~Han, X.~Wang, M.~Zhu, 
\emph{Characterization by symmetry of solutions of a nonlinear subelliptic equation on the Heisenberg group},
J. Math. Study \textbf{50} (2017), no. 1, 17--27.

\bibitem{JL87}
D.~Jerison and J.~M.~Lee,
\emph{The Yamabe problem on CR manifolds},
J. Differential Geom. \textbf{25} (1987), 167--197.

\bibitem{JL88}
D.~Jerison and J.~M.~Lee,
\emph{Extremals for the Sobolev inequality on the Heisenberg group and the CR Yamabe problem},
J. Amer. Math. Soc. \textbf{1} (1988), 1--13.

\bibitem{JL89}
D.~Jerison and J.~M.~Lee,
\emph{Intrinsic CR normal coordinates and the CR Yamabe problem},
J. Differential Geom. \textbf{29} (1989), 303--343.

\bibitem{LP87}
J.~M.~Lee and T.~H.~Parker,
\emph{The Yamabe problem},
Bull. Amer. Math. Soc. \textbf{17} (1987), 37--91.

\bibitem{Li04} 
Y.~Y.~Li, 
\emph{Remark on some conformally invariant integral equations:
the method of moving spheres}, 
J. Eur. Math. Soc. \textbf{6} (2004), 153--180.

\bibitem{LM12}
Y.~Y.~Li and D.~D.~Monticelli,
\emph{On fully nonlinear CR invariant equations on the Heisenberg group},
Journal of Differential Equations \textbf{252} (2012), 1309--1349.

\bibitem{LZ03}
Y.~Li and L.~Zhang,
\emph{Liouville-type theorems and Harnack-type inequalities for semilinear elliptic equations},
J. Anal. Math. \textbf{90} (2003), 27--87.

\bibitem{LZ95} 
Y.~Y.~Li, M.~Zhu, 
\emph{Uniqueness theorems through the method of moving spheres}, 
Duke Math. J. \textbf{80} (1995), 383--417.

\bibitem{MO23}
X.-N.~Ma and Q.~Ou,
\emph{A Liouville theorem for a class semilinear elliptic equations on the Heisenberg group},
Adv. Math. \textbf{413} (2023), Paper No.~108851, 20.

\bibitem{MM06}
R.~Monti and D.~Morbidelli,
\emph{Kelvin transform for Grushin operators and critical semilinear equations},
Duke Math. J. \textbf{131} (2006), no.~1.

\bibitem{Terracini95}
S.~Terracini,
\emph{Symmetry properties of positive solutions to some elliptic equations with nonlinear boundary conditions},
Differential Integral Equations \textbf{8} (1995), no.~8, 1911--1922.

\bibitem{Terracini96}
S.~Terracini,
\emph{On positive entire solutions to a class of equations with a singular coefficient and critical exponent},
Adv. Differential Equations \textbf{1} (1996), no.~2, 241--264.

\bibitem{Xu09}
L.~Xu,
\emph{Semilinear Liouville theorems in the Heisenberg group via vector field methods},
J. Differ. Equ. \textbf{247} (2009), 2799--2820.

\end{thebibliography}
\end{document}